\documentclass[english,a4paper,12pt]{amsart}
\usepackage[a4paper,lmargin=2cm,rmargin=2cm,tmargin=4cm,bmargin=4cm]{geometry}
\usepackage[centertags]{amsmath}
\usepackage{amsfonts}
\usepackage{amssymb}
\usepackage{amsthm}
\usepackage[draft]{graphicx}
\usepackage[colorlinks]{hyperref}

\newcommand{\Natural}{\mathbb N}

\newcommand{\ds}{\displaystyle}

\newcommand{\abs}[1]{\left\vert#1\right\vert}
\newcommand{\set}[1]{\left\{#1\right\}}
\newcommand{\restricted}{\upharpoonright}
\newcommand{\cardinality}[1]{\abs{#1}}

\newcommand{\dist}{\mathop{\mathrm{dist}}\nolimits}

\newcommand{\Sz}{\mathop{\mathrm{Sz}}\nolimits}

\newcommand{\norm}[1]{\left\Vert#1\right\Vert}
\newcommand{\duality}[1]{\left\langle#1\right\rangle}

\newcommand{\embeds}[1]{\mathop{\hookrightarrow}_{#1}}
\newcommand{\basepoint}{{\mathbf 0}}

\theoremstyle{plain}
\newtheorem{thm}{Theorem}
\newtheorem{cor}[thm]{Corollary}
\newtheorem{lem}[thm]{Lemma}
\newtheorem{prop}[thm]{Proposition}

\newtheorem*{Zippin}{Zippin's lemma}

\theoremstyle{definition}
\newtheorem{defn}[thm]{Definition}
\newtheorem{rem}[thm]{Remark}

\begin{document}
\title{Low distortion embeddings between $C(K)$ spaces}
\author{Anton\'\i n Proch\'azka$^\dag$}
\address{$^\dag$ Universit\'e Franche-Comt\'e\\
Laboratoire de Math\'ematiques UMR 6623\\
16 route de Gray\\
25030 Besan\c con Cedex\\
France}
\email{antonin.prochazka@univ-fcomte.fr}

\author{Luis S\'anchez-Gonz\'alez$^\ddag$}

\address{$^\ddag$ Departamento de Ingenier{\'i}a Matem{\'a}tica\\ Facultad de CC. F{\'i}sicas y  Matem{\'a}ticas\\ Universidad de Concepci{\'o}n\\ Casilla 160-C, Concepci{\'o}n, Chile}
\thanks{
 The second named author was partially supported by   MICINN Project MTM2012-34341 (Spain) and FONDECYT project 11130354 (Chile). This work started while L. S\'anchez-Gonz\'alez held a post-doc position at Universit\'e Franche-Comt\'e}
\email{lsanchez@ing-mat.udec.cl}

\begin{abstract}
We show that, for each ordinal $\alpha<\omega_1$, the space $C([0,\omega^\alpha])$ does not embed into $C(K)$ with distortion strictly less than $2$ unless $K^{(\alpha)}\neq \emptyset$.
\end{abstract}
\maketitle
\section{Introduction}
A mapping $f:M \to N$ between metric spaces $(M,d)$ and $(N,\rho)$ is called \emph{Lipschitz embedding} if there are constants $C_1,C_2>0$ such that
$C_1 d(x,y)\leq \rho(f(x),f(y))\leq C_2 d(x,y)$ for all $x,y \in M$. 
The distortion $\dist(f)$ of $f$ is defined as $\inf \frac{C_2}{C_1}$ where the infimum is taken over all constants $C_1,C_2$ which satisfy the above inequality. 
We say that $M$ embeds into $N$ with distortion $D$ (in short $\ds M \mathop{\hookrightarrow}_D N$) if there exists a Lipschitz embedding $f:M\to N$ with $\dist(f)\leq D$ (in short $f:\ds M \mathop{\hookrightarrow}_D N$). 
In this case, if the target space $N$ is a Banach space, we may always assume (by changing $f$) that $C_1=1$. 
The $N$-distortion of $M$ is defined as $c_N(M):=\ds \inf\set{D:M \mathop{\hookrightarrow}_D N}$.
The main result of this article (Theorem~\ref{t:classification}) implies in particular that, for every countable ordinal $\alpha$ and every $\beta<\omega^\alpha$, the space $C([0,\omega^\alpha])$ does \emph{not} embed into the space $C([0,\beta])$ with distortion strictly less than $2$. 
On the other hand, it has been shown by Kalton and Lancien~\cite{KL} that every separable metric space embeds into $c_0$ with distortion $2$. 
Since every $C(K)$ contains $c_0$ as a closed subspace, our result gives $c_{C([0,\beta])}(C([0,\omega^\alpha]))=2$ if $\beta<\omega^\alpha$.

It is a well known theorem of Mazurkiewicz and Sierpi\'nski (see~\cite[Theorem~2.56]{HMVZbook}) that every countable ordinal interval $[0,\beta]$ (every countable Hausdorff compact in fact) is homeomorphic to the interval $[0,\omega^\alpha\cdot n]$ where for some $\alpha<\omega_1$ and $1\leq n <\omega$.
Thus the corresponding spaces of continuous functions are isometrically isomorphic.
We do not know, whether one has $c_{C([0,\omega^\alpha\cdot m])}(C([0,\omega^\alpha\cdot n]))=2$ for $1\leq m<n<\omega$ but we get as a byproduct of the proof of our main result  that, for every $1\leq D<2$ and every $1\leq m <\omega$, there is $1\leq n<\omega$ such that for all $\alpha<\omega_1$ the space $C([0,\omega^\alpha\cdot n])$ does not embed into the space $C([0,\omega^\alpha\cdot m])$ with distortion $D$ (Proposition~\ref{p:byproduct}).

Metric spaces $M$ and $N$ are called \emph{Lipschitz homeomorphic} if there is a surjective Lipschitz embedding from $M$ \emph{onto} $N$. Such embedding is then called \emph{Lipschitz homeomorphism}.
The theorem of Amir~\cite{Amir} and Cambern~\cite{Cambern} is the following generalization of Banach-Stone theorem: 
Let $K$ and $L$ be two compact spaces. 
If there exists a linear isomorphism $f:C(K) \to C(L)$ such that $\dist(f)<2$, then $K$ and $L$ are homeomorphic. 
A result of Cohen~\cite{Cohen} shows that the constant $2$ above is optimal but at the present it is not clear whether one could draw the same conclusion under the weaker hypothesis of $f:C(K) \to C(L)$ being a Lipschitz homeomorphism such that $\dist(f)<2$.
The results of Jarosz~\cite{Jarosz}, resp. Dutrieux and Kalton~\cite{DK}, resp. G\'orak~\cite{Gorak}, show that $K$ and $L$ are homeomorphic if there is a Lipschitz homeomorphism $f:C(K) \to C(L)$ such that $\dist(f)<1+\varepsilon$ ($\varepsilon>0$ universal but small), resp. $\dist(f)<17/16$, resp. $\dist(f)<6/5$.
The main result of this article also implies that if $K$ and $L$ are two 
countable compacts,
then assuming the existence of a Lipschitz homeomorphism $f:C(K) \to C(L)$ with $\dist(f)<2$ implies that $K$ and $L$ have the same height (Corollary~\ref{c:AnswerGorak}).

The proof of the main result consists in identifying a uniformly discrete subset of $C([0,\omega^\alpha])$ which does not embed into $C([0,\beta])$ with distortion strictly less than $2$ if $\beta<\omega^\alpha$. The uniform discreteness allows to translate the above results into the language of uniform and net homeomorphisms (Corollary~\ref{c:AnswerGorak}). 
Also, using Zippin's lemma,  it permits to prove that if the Szlenk index of a Banach space $X$ satisfies $\Sz(X)\leq\omega^{\alpha}$, then $C([0,\omega^{\omega^\alpha}])$ does not embed into $X$ with distortion strictly less than $2$ (Theorem~\ref{t:Szlenk}).
This can be understood as a refinement of our result in~\cite{ProSan} that $C([0,1])$ does not embed into any Asplund space with distortion strictly less than $2$. 
The underlying idea of the proof there is basically the same as here (and originates in fact in~\cite{Aharoni} and in~\cite{Baudier}) but the proof is not obscured by the technicalities that are necessary in the present article.
 
Finally, let us recall the following result of Bessaga and Pe\l czy\'nski~\cite{BessagaPelczynski1960,HMVZbook}: 
Let $\omega\leq \alpha\leq \beta<\omega_1$. 
Then $C([0,\alpha])$ is linearly isomorphic to $C([0,\beta])$ iff $\beta<\alpha^\omega$. 
It is a longstanding open problem whether $C([0,\beta])$ can be Lipschitz homeomorphic to a \emph{subspace} of $C([0,\alpha])$ if $\beta>\alpha^\omega$.
Our results described above imply that the distortion of any such Lipschitz homeomorphism onto a subspace must be at least $2$. 

Besides this Introduction, the paper features two more sections. 
In Section~\ref{s:Main} the main theorem is stated and proved. 
In Section~\ref{s:Concluding} we state and prove its various consequences.


\section{Main theorem}\label{s:Main}

%
\begin{thm}\label{t:classification}
For every ordinal $\alpha<\omega_1$ there exists a countable uniformly discrete metric space $M_\alpha\subset C([0,\omega^\alpha])$ such that $M_\alpha$ does not embed with distortion strictly less than $2$ into $C(K)$ if $K^{(\alpha)}=\emptyset$. 
\end{thm}

We start by defining finite 
metric graphs that do not embed well into $\ell_\infty^n$ if $n$ is small.
We then ``glue'' them together infinitely many times, via a relatively natural procedure that we call sup-amalgamation. The sup-amalgamation is done in a precise order which will be encoded by certain trees on $\Natural$. 

\subsection{Construction of 3-level metric graphs}
\begin{defn}
Let $C_1,\ldots,C_h$ be pairwise disjoint sets. 
We put 
\[
 \begin{split}
  M(C_1,\ldots,C_h)&=\set{\basepoint} \mbox{ first level}\\
&\cup \bigcup_{i=1}^h C_i \mbox{ second level}\\
&\cup F(C_1,\ldots,C_h) \mbox{ third level}
 \end{split}
\]
where $F(C_1,\ldots,C_h)=\set{\set{c_1,\ldots,c_h}:c_i \in C_i}$.
We turn $M(C_1,\ldots,C_h)$ into a graph by putting an edge between $x,y \in M(C_1,\ldots,C_h)$ iff $x=\basepoint$ and $y \in \bigcup C_i$ or $x \in \bigcup C_i$, $y \in F(C_1,\ldots,C_h)$ and $x \in y$. We consider the shortest path distance $d$ on $M(C_1,\ldots,C_h)$.
\end{defn}

\begin{lem}\label{l:embed-finite}
Let $C_1,\ldots,C_h$ be pairwise disjoint sets and let us assume that $C_1=\set{1,2}$. We denote $F=F(C_1,\ldots,C_h)$.
Then there is an isometric embedding $f:M(C_1,\ldots,C_h) \to \ell_\infty(F)$ which satisfies
\begin{itemize}
\item $f(\basepoint)=0$,
\item $f(x)(\beta)\in \set{\pm 1}$ for all $x \in \bigcup_{i=1}^h C_i$ and  all $\beta \in F$,
\item $f(1)(\beta)=1$ and $f(2)(\beta)=-1$ for all $\beta \in F$.
\end{itemize}
\end{lem}
\begin{proof}
We define first a mapping $g:M(C_1,\ldots,C_h) \to \ell_\infty(F)$ as $g(x)(\beta)=d(x,\beta)-d(\basepoint,\beta)$. It is clearly $1$-Lipschitz. Given $x,y \in M(C_1,\ldots,C_h)$, we can always find $A,B \in F$ such that $d(A,B)=d(A,x)+d(x,y)+d(y,B)$. 
We have $g(x)(B)-g(y)(B)=d(x,B)-d(y,B)=d(x,B)+d(x,A)-d(A,B)+d(x,y)\geq g(x,y)$.
So $g$ is an isometry. 
Observe that $g(\basepoint)=0$ and also $g$ satisfies the second additional property.
Now since $1 \in \beta$ iff $2 \notin \beta$ for any $\beta \in F$, we have that $g(1)(\beta)=+1$ iff $g(2)(\beta)=-1$ for all $\beta \in F$.
We thus define 
\[
f(x)(\beta):=
\begin{cases}
g(x)(\beta)&\mbox{if } g(1)(\beta)=1\\
-g(x)(\beta)&\mbox{if } g(1)(\beta)=-1
\end{cases}
\]
for all $x \in M(C_1,\ldots,C_h)$.
\end{proof}


\subsection{Sup-amalgam of metric spaces}
\begin{defn}
Let $((M_i,d_i))_{i\in I}$ be a collection of metric spaces of uniformly bounded diameter. Let us assume that there is a set $A$ and a distinguished point $\basepoint \in A$ such that $A \subset M_i$ for every $i \in I$. A \emph{product} $(P,d)$ is the set $\displaystyle P=\prod_{i \in I} M_i$ equipped with the metric $\displaystyle d(x,y)=\sup_{i \in I} d_i(x(i),y(i))$.
The set $M_A \subset P$ defined by $x \in M_A$ iff either $x(i)=x(j) \in A$ for all $i,j \in I$ or there exists exactly one $i \in I$ such that $x(i) \notin A$ and for all $j\neq i$ we have $x(j)=\basepoint$, equipped with the metric $d$, is called the \emph{sup-amalgam of $(M_i)$ with respect to $A$}. We denote it $(M_i)_{i\in I}/A$. 
\end{defn} 
\noindent
{\bf Standing assumption SA1:} Even though the definition admits the possibility that $d_i$ and $d_j$ for $i\neq j$ are different on $A$, in what follows we will always assume that $d_i\restricted_{A\times A}=d_j\restricted_{A\times A}$. In that case there is a canonical isometric copy of $A$ in $M_A$ which we will denote by $A$ again.\\
{\bf Standing assumption SA2:} We will also assume from now on that for each $i \in I$ we have that $d_i(x,y)\geq 1$ for all $x,y \in M_i$ and $d_i(x,\basepoint)\leq 1$ for each $x \in A$. Then, for each $i\in I$, the canonical copy of $M_i$ in $M_A$ is isometric to $M_i$.\\
Proof: We only need to show that for $x \in A$ and $y \in M_i \setminus A$ we have $d_i(x,y)=d(x,y)$. This is equivalent to saying that $d_j(x,\basepoint)\leq d_i(x,y)$ for all $j \in I$.
\qed

\begin{lem}\label{l:embed-sup-amalgam}
a) Let $A$ be a finite set. Let $(M_n)_{n \in \Natural}$ be a sequence of metric spaces of uniformly bounded diameter such that $\basepoint \in A \subset M_n$ for every $n \in \Natural$. We assume SA1 and SA2. If for each $n \in \Natural$ there is an ordinal $\alpha_n < \omega_1$ and an isometric embedding $f_n:M_n \to C([0,\alpha_n])$ so that for each $n \in \Natural$ we have
\begin{itemize}
\item $f_n(\basepoint)=0$,
\item $f_n(x)(\beta) \in \set{\pm 1}$ for each $x \in A\setminus\set{\basepoint}=:A_*$ and $\beta \in [0,\alpha_n]$,
\end{itemize}
then there are $N\leq 2^{\cardinality{A_*}}$ and an isometric embedding $f:M_A \to \displaystyle C([0,(\sum_{n=1}^\infty \alpha_n)\cdot N])$ such that $f(\basepoint)=0$ and $f(x)(\beta) \in \set{\pm 1}$ for each $x \in A_*$ and $\beta \in \displaystyle [0,(\sum_{n=1}^\infty \alpha_n)\cdot N]$.

b) Let us assume moreover that $1,2 \in A$ and that for every $n \in \Natural$ we have 
$f_n(1)(\beta)=1$ and $f_n(2)(\beta)=-1$ for all $\beta \in [0,\alpha_n]$. 
Then we have $f(1)(\beta)=1$ and $f(2)(\beta)=-1$ for all $\beta \in \ds [0,(\sum_{n=1}^\infty \alpha_n)\cdot N]$.

c) Finally, assume moreover that $A=\set{\basepoint,1,2}$. Then $N=1$.
\end{lem}

\begin{proof}
Let us consider the restriction $g$ of the product mapping 
\[\displaystyle
\prod_{n=1}^\infty M_n \ni x \mapsto 
(f_n(x(n)))_{n=1}^\infty \in \left(\bigoplus_{n=1}^\infty C([0,\alpha_n])\right)_\infty
\]
to the set $M_A$. The mapping $g$ is clearly an isometry, $g(\basepoint)=0$ and for each $x \in M_A \setminus A$ we have $g(x) \in \displaystyle C([0,\sum_{n=1}^\infty \alpha_n])$ as it has exactly one non-zero entry. 
Notice that $g(a) \in \displaystyle C([0,\sum_{n=1}^\infty \alpha_n))$ for $a\in A$. 
Now, for each $\varepsilon \in \set{\pm 1}^{A_*}$ we consider the set
$T_\varepsilon=\displaystyle\bigcap_{a \in A_*} \set{\beta\in \left[0,\sum_{n=1}^\infty \alpha_n\right): g(a)(\beta)=\varepsilon(a)}$. 
Let us consider $I=\set{\varepsilon \in \set{\pm 1}^{A_*}: T_\varepsilon \neq \emptyset}$. 
The sets $(T_\varepsilon)_{\varepsilon \in I}$ are a disjoint cover of $\displaystyle \left[0,\sum_{n=1}^\infty \alpha_n\right)$. 
The set $T_\varepsilon \cap [0,\eta]$ is clopen for each $\eta<\sum \alpha_n$ and each $\varepsilon \in I$. 
Let $\displaystyle(J_\varepsilon)_{\varepsilon \in I}$ be mutually disjoint copies of $\displaystyle [0,\sum_{n=1}^\infty \alpha_n]$, say $J_\varepsilon=[0,\sum_{n=1}^\infty \alpha_n]\times \set{\varepsilon}$.
If $x \in A_*$, we define $f(x)$ as the continuous function on $\bigcup_{\varepsilon \in I}J_\varepsilon$ such that  $f(x)(\beta)=\varepsilon(x)$ when $\beta \in J_\varepsilon$.
If $x \in M_A\setminus A_*$ we choose $\eta<\sum \alpha_n$ such that $g(x)(\gamma)=0$ for $\gamma>\eta$ and we define $f(x)$ as the continuous function on $\bigcup_{\varepsilon\in I} J_\varepsilon$
such that for each $\varepsilon \in I$ we have
\[
f(x)((\beta,\varepsilon))=
\begin{cases}
g(x)(\beta)&\mbox{ when } \beta \in T_\varepsilon \cap [0,\eta]\\
0&\mbox{ otherwise}.
\end{cases}
\]
Notice that we have $f(\basepoint)=0$ and $f(x)(\beta) \in \set{\pm 1}$ for $x \in A_*$ and $\beta \in \bigcup_{\varepsilon \in I} J_\varepsilon$.
Let us check that $f$ is an isometry. 
Using that $g$ is an isometry and the definition of $f$ and $T_\varepsilon$, it is obvious that 
\[
d(x,y) = \norm{g(x)-g(y)}=\sup\set{\abs{f(x)(\beta)-f(y)(\beta)}:\beta \in \bigcup_{\varepsilon \in I} \overline{T_\varepsilon\times\set{\varepsilon}}}.
\]
On the other hand, checking the four possibilities ($x \in A_*$ or $x \notin A_*$) and ($y \in A_*$ or $y \notin A_*$), and remembering SA2, we see that $\abs{f(x)(\beta)-f(y)(\beta)}\leq d(x,y)$ if $\beta \notin \bigcup_{\varepsilon \in I}\overline{T_\varepsilon\times\set{\varepsilon}}$. 
Thus, $f$ is an isometry from $M_A$ into $C\left(\displaystyle\bigcup_{\varepsilon \in I} J_\varepsilon\right)$.
It is clear that $\displaystyle\bigcup_{\varepsilon \in I} J_\varepsilon$ is homeomorphic to $\displaystyle [0,(\sum_{n=1}^\infty \alpha_n)\cdot N]$ where $N:=\cardinality{I}$. Hence $f$ maps isometrically $M_A$ into $C(\displaystyle [0,(\sum_{n=1}^\infty \alpha_n)\cdot N])$.

The hypothesis in b) means that $\varepsilon(1)=1$ and $\varepsilon(2)=-1$ for every $\varepsilon \in I$. 
Now the definition of $f$ on $A_*$ gives the conclusion of~b).

The hypothesis in b) and c) mean that $I=\set{\varepsilon}$ where $\varepsilon(1)=1$, $\varepsilon(2)=-1$. So $N=1$.
\end{proof}

Given a metric space $M$, a mapping $f: M \to C(K)$, points $a,b \in M$ and a constant $1\leq D<2$  we denote
\[
X_{a,b}^f:=\set{x^* \in K: \abs{\duality{x^*,f(a)-f(b)}}\geq 4-2D}.
\]
The duality above means the evaluation at the point $x^* \in K$.
We do not indicate the dependence on $D$ since it will always be clear from the context, which $D$ we have in mind.

\begin{lem}\label{l:nonemptyintersection}
Let $C_1,\ldots,C_h$ be pairwise disjoint finite sets, we denote $A=\ds\set{\basepoint} \cup\bigcup_{i=1}^h C_i$. 
Let $(M_n)_{n \in \Natural}$ be a sequence of metric spaces of uniformly bounded diameter such that $A \subset M_n$ for every $n \in \Natural$ (and they satisfy SA1 and SA2). 
Let $1\leq D<2$ and let $K$ be a Hausdorff compact. 
We assume that for each $n \in \Natural$ there is an ordinal $\alpha_n < \omega_1$ and an integer $\beta_n \in \Natural$ such that every $f_n:\ds M_n \mathop{\hookrightarrow}_D C(K)$ satisfies
\[
\cardinality{\bigcap_{i=1}^h X^{f_n}_{a_i,b_i} \cap K^{(\alpha_n)}} \geq{\beta_n}
\]
for all $a_i\neq b_i \in C_i \subset M_n$. 
Then the sup-amalgam $M_A=(M_n)_{n=1}^\infty/A$ satisfies for every $f:\ds M_A \mathop{\hookrightarrow}_D C(K)$,
\begin{itemize}
\item[{\bf a)}] if $\alpha_m<\lim_n \alpha_n=\alpha$ for all $m \in \Natural$, then
\[
\bigcap_{i=1}^h X^f_{a_i,b_i} \cap K^{(\alpha)} \neq \emptyset
\]
for all $a_i\neq b_i \in C_i \subset M_A$.
\item[{\bf b)}] if $\alpha_n=\alpha$ for all $n\in \Natural$, then 
\[
\cardinality{\bigcap_{i=1}^h X^f_{a_i,b_i} \cap K^{(\alpha)}} \geq \sup \beta_n
\] 
for all $a_i\neq b_i \in C_i \subset M_A$.
\end{itemize}
\end{lem}

\begin{proof}
Let $\ds f:M_A \embeds{D} C(K)$. Then $f_n:=f\restricted_{M_n}$ is an embedding of $M_n$ into $C(K)$ with $\dist(f_n)\leq D$, and the conclusion follows.
\end{proof}

Clearly, we do not cover all the possibilities in the above lemma, but these two are the interesting for us. 

\subsection{Iterative construction of $M_\alpha$}
We use trees here in a very basic fashion as index sets. 
For the notation, check~\cite{HMVZbook}.
For two finite sequences $\overline{m}=(m_1,\ldots,m_h)$ and $\overline{n}=(n_1,\ldots,n_l)$, we write $\overline{m}^\smallfrown \overline{n}=(m_1,\ldots,m_h,n_1,\ldots,n_l)$ for their concatenation.
We omit the parentheses for the sequences of length one, thus $(n)$ is written as $n$.
Let us construct for every ordinal $\alpha<\omega_1$ the tree $T_{\alpha+1}$  on $\Natural$ as follows: 
\begin{itemize}
\item $T_1=\Natural$
\item if $\alpha$ is non-limit, we put $\ds T_{\alpha+1}=\Natural \cup \bigcup_{n=1}^\infty n^\smallfrown T_\alpha$ 
\item if $\alpha$ is limit, we choose some $\alpha_n \nearrow \alpha$ and put $\ds T_{\alpha+1}=\Natural \cup \bigcup_{n=1}^\infty n^\smallfrown T_{\alpha_n+1}$ 
\end{itemize}
where $n^\smallfrown T_\alpha = \{n^\smallfrown\overline{m}: \overline{m} \in T_\alpha\}$. 
Clearly, for each $\alpha$, the tree $T_{\alpha+1}$ is well founded.

Further, let us define a derivation on trees as follows:
\[
T'=\set{(n_1,\ldots,n_h) \in T: (n_1,\ldots,n_{h-1},k) \mbox{ is not maximal in } T \mbox{ for some }k\in \Natural}
\]
There is an index $o$ naturally tied to this derivation. We put $T^{(0)}=T$, $T^{(\alpha+1)}=(T^{(\alpha)})'$ and $T^{(\alpha)}=\bigcap_{\beta<\alpha} T^{(\beta)}$ whenever $\alpha$ is a limit ordinal. We put $o(T)=\inf\set{\alpha: T^{(\alpha)}=\emptyset}$ if the set is nonempty, otherwise $o(T)=\infty$.
We do not intend to characterize the trees for which $o(T)<\infty$.
Instead, we will compute the index of the trees $T_{\alpha+1}$ above.
It is worth noting that for every $\alpha, \beta$ and $\overline{n} \in T_{\alpha+1}^{(\beta)}$ the set of successors of $\overline{n}$ in $T_{\alpha+1}^{(\beta)}$ is either empty or equal to $\set{\overline{n}^\smallfrown k:k \in \Natural}$.
\begin{lem}\label{l:height}
For each $\alpha<\omega_1$ we have $o(T_{\alpha+1})=\alpha+1$.
\end{lem}
\begin{proof}
We have clearly $o(T_1)=1$ since all the nodes are maximal, and are mutual siblings.
Assume the claim to be true for all $\beta < \alpha$, we try to prove it for $\alpha$. 
If $\alpha=\beta+1$ is non-limit,  we have  
$\ds T_{\alpha+1}=\Natural \cup \bigcup_{n=1}^\infty n^\smallfrown T_{\beta+1}$. It is thus clear that 
$\ds T_{\alpha+1}^{(\beta+1)}=T_1$ and so $o(T_{\alpha+1})=\beta+1+1=\alpha+1$.
Finally assume that $\alpha$ is limit, we have $\ds T_{\alpha+1}=\Natural \cup \bigcup_{n=1}^\infty n^\smallfrown T_{\alpha_n+1}$. 
It is thus clear that $\ds T_{\alpha+1}^{(\alpha_m+1)}=\Natural \cup \bigcup_{n=m+1}^\infty n^\smallfrown T_{\alpha_n+1}^{(\alpha_m+1)}$ for all $m \in \Natural$.
Hence $\ds T_{\alpha+1}^{(\alpha)}=T_1$, and so $o(T_{\alpha+1})=\alpha+1$.
\end{proof}

For each node of a tree $T_{\alpha+1}$ it will be important to know how many derivations it  takes till the node becomes maximal. 

\begin{lem}\label{l:MaximalAtNonlimit}
For each $\alpha<\omega_1$ and each $\overline{n} \in T_{\alpha+1}$ 
the ordinal $r_{\alpha+1}(\overline{n}):=\inf\set{\beta:\overline{n} \in \max T_{\alpha+1}^{(\beta)}}$ is isolated ($0$ or successor).
\end{lem}

\begin{proof}
Induction on $\alpha$. Clearly true for $\alpha=0$ (we have $r_1(\overline{n})=0$ for each $\overline{n} \in T_1$).
Assume the claim to be true for $\alpha$. Then
$r_{\alpha+1+1}(n)=o(T_{\alpha+1})=\alpha+1$ for all $n\in \Natural$ by the construction of the tree $T_{\alpha+2}$ and by Lemma~\ref{l:height}. Further, for each $n\in \Natural$ and each $\overline{n}\in T_{\alpha+1}$, the ordinal $r_{\alpha+1+1}(n^\smallfrown\overline{n})=r_{\alpha+1}(\overline{n})$ is isolated by the inductive hypothesis.
In the case when $\alpha$ is limit and the claim has been proved for all $\beta<\alpha$,  
we have $r_{\alpha+1}(n)=o(T_{\alpha_n+1})=\alpha_n+1$ for all $n\in \Natural$ by the construction of the tree $T_{\alpha+1}$ and by Lemma~\ref{l:height}.
Further, for each  $n\in \Natural$ and each $\overline{n} \in T_{\alpha_n+1}$, the ordinal
$r_{\alpha+1}(n^\smallfrown\overline{n})=r_{\alpha_n+1}(\overline{n})$ is isolated by the inductive hypothesis.
\end{proof}


\begin{defn}\label{d:metric}
Let $(A_i)_{i=0}^\infty$ be a sequence of countably infinite pairwise disjoint sets which will be fixed from now on. 
Suppose that $A_i=\set{a^i_1,a^i_2,\ldots}$ and denote $A_i^k=\set{a^i_1,\ldots,a^i_k}$ for all $i \geq 0$ and $k \geq 1$. We assume that $A_0^2=\set{1,2}$.
To shorten the notation we put 
\[
M^0_{\overline{n}}=M(A_0^2,A_1^{n_1},\ldots,A_h^{n_h})
\]
for each $\overline{n}=(n_1,\ldots,n_h) \in \Natural^h$.

Let us fix $\mu<\omega_1$ from now on. The tree $T=T_{\mu+1}$ encodes a construction of a metric space.
First we consider all the maximal elements $\max T$ of $T$.
We thus have a collection ${\mathcal M}_0=\set{M^0_{\overline{n}}:\overline{n} \in \max T}$. 
Each space in this collection is finite.
Once ${\mathcal M}_\alpha=\set{M^\alpha_{\overline{n}}:\overline{n} \in \max T^{(\alpha)}}$ has been defined, 
we pass to the collection ${\mathcal M}_{\alpha+1}=\set{M^{\alpha+1}_{\overline{n}}:\overline{n} \in \max T^{(\alpha+1)}}$ as follows.
For every $\overline{n}=(n_1,\ldots,n_h) \in \max T^{(\alpha+1)}$ we put 
\[
M^{\alpha+1}_{\overline{n}}=
\begin{cases}(M^{\alpha}_{\overline{n}^\smallfrown k})_{k=1}^\infty/(\set{\basepoint}\cup A_0^2 \cup A_1^{n_1}\cup\ldots\cup A_{h}^{n_h})&\mbox{ if }\overline{n} \in \max T^{(\alpha+1)}\setminus \max T^{(\alpha)}\\
M_{\overline{n}}^\alpha&\mbox{ if }\overline{n} \in \max T^{(\alpha+1)} \cap \max T^{(\alpha)}.
\end{cases}
\]
When $\alpha<\mu$ is a limit ordinal, we define the elements of ${\mathcal M}_\alpha=\set{M^\alpha_{\overline{n}}:\overline{n} \in \max T^{(\alpha)}}$ as $\displaystyle M^\alpha_{\overline{n}}=\lim_{\beta<\alpha} M^\beta_{\overline{n}}$. 
This definition makes sense since for every $\overline{n} \in \max T^{(\alpha)}$ there is $\beta_0<\alpha$ such that $\overline{n} \in \max T^{(\beta)}$ for all $\beta_0\le \beta  < \alpha$. Indeed, by Lemma~\ref{l:MaximalAtNonlimit}, we may take $\beta_0=r_{\mu+1}(\overline{n})$. It is isolated, so $\beta_0<\alpha$.

At the end of the day we have ${\mathcal M}_\mu=\set{M^\mu_n:n \in \Natural}$ which we glue into a single space $M_\emptyset^{\mu+1}=(M^\mu_n)_{n=1}^\infty/(\set{\basepoint}\cup A_0^2)$.
\end{defn}

\begin{lem}\label{l:embedding}
For every $1\leq \alpha \leq \mu+1$ for every $\overline{n}=(n_1,\ldots,n_h) \in \max T^{(\alpha)}$ (or rather $\overline{n}=\emptyset$ in the case $\alpha=\mu+1$) there are $N\leq 2^{2+n_1+\ldots + n_h}$ and an isometric embedding  $f:M^\alpha_{\overline{n}} \to C([0,\omega^\alpha\cdot N])$ such that $f(\basepoint)=0$, $f(1)\equiv 1$ and $f(2)\equiv-1$. When $\alpha=\mu+1$, we have $N=1$.
\end{lem}
\begin{proof}
By the definition of the spaces in ${\mathcal M}_0$, we get the claim for $\alpha=1$ using Lemma~\ref{l:embed-finite} and Lemma~\ref{l:embed-sup-amalgam} a)-b) (here the ordinals $\alpha_n$ are finite so the spaces $C([0,\alpha_n])$ and $\ell_\infty(\alpha_n)$ are isometric). 
For $\alpha>1$, the proof is a standard transfinite induction argument exploiting Lemma~\ref{l:embed-sup-amalgam} a)-b).
Finally, when $\alpha=\mu+1$, the definition of $M_\emptyset^{\mu+1}$ shows that also the part c) of Lemma~\ref{l:embed-sup-amalgam} applies.
\end{proof}

\begin{lem}\label{l:precise}
(i) For each $1\leq D<2$ there exists a constant $C_D=\left(\log\left(\lfloor\frac{D}{2-D}\rfloor+1\right)\right)^{-1}>0$ such that  for every $\alpha\leq \mu$ and every $\overline{n}=(n_1,\ldots,n_h) \in \max T^{(\alpha)}$ and every $f:M^\alpha_{\overline{n}}\ds\mathop{\hookrightarrow}_D C(K)$ we have
\[
 \cardinality{X_{1,2}^f\cap \bigcap_{i=1}^{h-1} X^f_{a_i,b_i} \cap K^{(r_{\mu+1}(\overline{n}))}} \geq C_D \log(n_h)
\]
for all $a_i\neq b_i \in A_i^{n_i}$, $1\leq i\leq h-1$ (with the obvious meaning when $h=1$).

(ii) For each $1\leq D<2$ and every $\ds f:M^{\mu+1}_\emptyset \embeds{D} C(K)$ we have
\[
X_{1,2}^f\cap K^{(\mu+1)} \neq \emptyset.
\]
\end{lem}
\begin{proof}
We will proceed by a transfinite induction on $\alpha$. 
We assume the result to be true for every $\beta<\alpha$ and we want to prove it for $\alpha$.
Clearly it is enough to prove the claim for $\alpha=r_{\mu+1}(\overline{n})$. 
Indeed, if $\alpha>r_{\mu+1}(\overline{n})$ then $M_{\overline{n}}^{\alpha}$ and $M_{\overline{n}}^{r_{\mu+1}(\overline{n})}$ are the same, and the result follows by the inductive hypothesis.

In the case $0<\alpha=r_{\mu+1}(\overline{n})$, the node $\overline{n}$ just became maximal, i.e. $\overline{n} \in \max T^{(\alpha)}\setminus \max T^{(\beta)}$ where $\alpha=\beta+1$ (remember Lemma~\ref{l:MaximalAtNonlimit}). 
This means that the immediate successors $\set{\overline{n}^\smallfrown k :k \in \Natural}$ of $\overline{n}$ in $T^{(\beta)}$ are all maximal in $T^{(\beta)}$ and
\[
M_{\overline{n}}^\alpha=(M_{\overline{n}^\smallfrown k}^\beta)_{k=1}^\infty/(\set{\basepoint} \cup A_0^2\cup A_1^{n_1}\cup\ldots\cup A_h^{n_h}).
\]
By the inductive hypothesis we have for every $k \in \Natural$ and every $\ds f_k:M_{\overline{n}^\smallfrown k}^\beta \embeds{D} C(K)$ that
\begin{equation}\label{e:logarithmic-lower-bound}
 \cardinality{X^{f_k}_{1,2}\cap \bigcap_{i=1}^{h} X^{f_k}_{a_i,b_i} \cap K^{(r_{\mu+1}(\overline{n}^\smallfrown k))}} \geq C_D \log(k)
\end{equation}
for every choice $a_i\neq b_i \in A_i^{n_i}$, $1\leq i \leq h$.

By the construction of $T_{\mu+1}$ it is clear that we have two types of non-maximal nodes $\overline{n} \in T_{\mu+1}$:\\ 
{\textbf a)}~those for which  $(r_{\mu+1}(\overline{n}^\smallfrown k))_{k \in \Natural}$ is a strictly increasing sequence of ordinals and\\ 
{\textbf b)} those for which it is a constant sequence.

The case a) means that $\beta$ is a limit ordinal
and $\sup\set{r_{\mu+1}(\overline{n}^\smallfrown k):k \in \Natural}+1=r_{\mu+1}(\overline{n})$  by the definition of $r_{\mu+1}(\overline{n})$.
The case b) means that $\beta$ is a successor and for all $k\in \Natural$ we have $r_{\mu+1}(\overline{n}^\smallfrown k)+1=r_{\mu+1}(\overline{n})$.


Let $f$ be an embedding such that $d(x,y) \leq \norm{f(x)-f(y)}\leq Dd(x,y)$. 
Remembering~\eqref{e:logarithmic-lower-bound}, we apply Lemma~\ref{l:nonemptyintersection} to get
\[
 X^f_{1,2}\cap \bigcap_{i=1}^{h} X^f_{a_i,b_i} \cap K^{(r_{\mu+1}(\overline{n}))} \neq \emptyset
\]
for every choice $a_i\neq b_i \in A_i^{n_i}$, $1\leq i \leq h$.
Notice that since $\sup\set{r_{\mu+1}(n):n\in \Natural}=\mu$, we can put formally $r_{\mu+1}(\emptyset)=\mu+1$, and so the above also proves (ii).

Now we need to pass from ``non-empty'' to ``larger than $C_D \log(n_h)$''. 
We may assume that $f(\basepoint)=0$.
Let $a_i \neq b_i \in A_i^{n_i}$ be fixed for $1 \leq i \leq h-1$. By the above there exists 
\[
x^*_{a,b} \in X_{1,2}^f\cap \bigcap_{i=1}^{h-1} X^f_{a_i,b_i} \cap X^f_{a,b}\cap K^{(r_{\mu+1}(\overline{n}))}
\]  
for each  $a\neq b \in A_h^{n_h}$. 
We set $\Gamma=\set{x^*_{a,b}:a,b \in A_h^{n_h}, a\neq b}$. Now 
\[
\set{(\duality{f(a),\gamma})_{\gamma \in \Gamma}:a \in A_h^{n_h}}\subset B_{\ell_\infty(\Gamma)}(0,D)
\] 
is $(4-2D)$-separated set of cardinality $n_h$. We thus get that $\cardinality{\Gamma}\geq C_D\log(n_h)$.

It remains to prove what happens if $\alpha=0$. So let $\ds f:M^0_{\overline{n}} \embeds{D} C(K)$ and let $a_i\neq b_i \in A_i^{n_i}$, $1\leq i \leq h$, be given. We put $A=\set{1}\cup\set{a_i:1\leq i\leq h}$ and $B=\set{2}\cup \set{b_i:1\leq i\leq h}$. We get by the triangle inequality that each $x^* \in K$ such that $\abs{\duality{x^*,f(A)-f(B)}}=\norm{f(A)-f(B)}$ satisfies
\[
 x^* \in X^f_{1,2}\cap \bigcap_{i=1}^{h} X^f_{a_i,b_i}.
\]
Now by the same argument as above 
we get the desired inequality.
\end{proof}
\begin{proof}[Proof of Theorem~\ref{t:classification}]
Let $\alpha<\omega_1$ be given. If it is isolated, say $\alpha=\mu+1$, we put $M_\alpha:=M^{\mu+1}_\emptyset$. This space embeds isometrically into $C([0,\omega^\alpha])$ by Lemma~\ref{l:embedding}. 
By Lemma~\ref{l:precise} (ii) we see that if $\ds M_\emptyset^{\mu+1} \embeds{D} C(K)$, $D<2$, then $K^{(\alpha)}\neq \emptyset$.

Finally, if $\alpha$ is a limit ordinal we choose $\mu_n \nearrow \alpha$ and we put 
\[
M_{\alpha}=(M^{\mu_n+1}_\emptyset)_{n=1}^\infty/(\set{\basepoint}\cup A_0^2).
\]
Since $\sum \mu_n=\alpha$ and since $M^{\mu_n+1}_\emptyset$ embeds isometrically into $C([0,\omega^{\mu_n+1}])$ for every $n \in \Natural$, Lemma~\ref{l:embed-sup-amalgam} shows that $M_{\alpha}$ embeds isometrically into $C([0,\omega^\alpha])$.
If $\ds f:M_\alpha \embeds{D} C(K)$, $D<2$, then $X^f_{1,2} \cap K^{(\mu_n+1)}\neq \emptyset$ for each $n \in \Natural$. This is a decreasing set of compact sets. Hence $X^f_{1,2} \cap K^{(\alpha)} \neq \emptyset$.
\end{proof}

\section{Concluding remarks}\label{s:Concluding}

\begin{rem}
Let $\alpha<\omega_1$ and let $K$ be a Hausdorff compact space such that $K^{(\alpha)}\neq \emptyset$. We denote by $C_0(K)$ the closed subspace of $C(K)$ of the functions whose restrictions on $K^{(\alpha)}$ are identically zero.
An inspection of above proof shows that $C([0,\omega^\alpha])$ does not embed with distortion strictly less than $2$ into $C_0(K)$. 
In particular, $C([0,\omega^\alpha])$ does not embed with distortion strictly less than $2$ into $C_0([0,\omega^\alpha])$.
For $\alpha=1$ the last statement means that $c$ does not embed with distortion strictly less than $2$ into $c_0$. This also follows from \cite[Proposition 3.1]{KL} as an easy but entertaining exercise.
\end{rem}

\begin{prop}\label{p:byproduct}
Let $1\leq D<2$ be given. Then for every $1\leq m <\omega$ there is $1\leq n<\omega$ such that for all $\alpha<\omega_1$ the space $C([0,\omega^\alpha\cdot n])$ does not embed into the space $C([0,\omega^\alpha\cdot m])$ with distortion $D$.
\end{prop}
\begin{proof}
We find $k \in \Natural$ such that $C_D\log(k)> m$ and we put $n=2^{2+k}$.
Suppose first that $\alpha$ is a successor ordinal.   
Notice that if we consider $k$ as an element of $\max T_{\alpha+1}^{(\alpha)}$, we have $r_{\alpha+1}(k)=\alpha$. 
By Lemma~\ref{l:precise}, the space $M_{k}^{\alpha}$ as defined in Definition~\ref{d:metric} does not embed with distortion $D$ into $C([0,\omega^\alpha\cdot m])$. 
On the other hand, it embeds isometrically into $C([0,\omega^\alpha\cdot n])$ by Lemma~\ref{l:embedding}.

If $\alpha$ is a limit ordinal, we replace $A_0^2$ by $A_0^k$ in Definition~\ref{d:metric} and we consider the space 
\[
\widetilde{M}_\alpha:=(M^{\alpha_l+1}_\emptyset)_{l=1}^\infty/(\set{\basepoint}\cup A_0^k).
\]
As above, we see that $\widetilde{M}_\alpha$ embeds isometrically into $C([0,\omega^\alpha\cdot n])$.
By repeating the proof of Lemma~\ref{l:precise} we get that $\widetilde{M}_\alpha$ does not embed with distortion $D$ into $C([0,\omega^\alpha\cdot m])$. We leave the details to the reader.
\end{proof}

We recall that if $X$ and $Y$ are Banach spaces and $u:X \to Y$ is uniformly continuous, the following \emph{Lipschitz constant of $u$ at infinity}
\[
l_\infty(u)=\inf_{\eta>0}\sup_{\norm{x-x'}\geq \eta} \frac{\norm{u(x)-u(x')}}{\norm{x-x'}}
\]
is finite (sometimes called Corson-Klee lemma, see~\cite[Proposition 1.11]{BLbook}).
The \emph{uniform distance} between $X$ and $Y$ is $d_U(X,Y)=\inf l_\infty(u)\cdot l_\infty(u^{-1})$, where the infimum is taken over all uniform homeomorphisms between $X$ and $Y$.

A \emph{net} in a Banach space $X$ is a subset $\mathcal N$ of $X$ such that there exist $a,b>0$ which satisfy
\begin{itemize}
\item  for any $x,x' \in \mathcal N$ with $x\neq x'$, we have $\norm{x-x'}\geq a$ and,
\item  for any $x \in X$, there exists $y \in \mathcal N$ with $\norm{x-y}\leq b$.
\end{itemize}
We say that two Banach spaces are \emph{net-equivalent} when they have Lipschitz homeomorphic nets. The \emph{net distance} between $X$ and $Y$ is the number $d_N(X,Y)=\inf \dist(f)$ where the infimum is taken over all mappings $f:\mathcal N \to \mathcal M$ with $\mathcal N \subset X$ and $\mathcal M \subset Y$ being nets. Finally $d_L(X,Y)=\inf \dist(f)$ where the infimum is taken over all Lipschitz homeomorphisms $f:X \to Y$. 
It is well known and easy to see that for any couple of Banach spaces $X$ and $Y$ we have
\[
d_N(X,Y) \leq d_U(X,Y) \leq d_L(X,Y).
\]
\begin{cor}\label{c:AnswerGorak}
Let $\gamma\neq \alpha<\omega_1$ and $n,m \in \Natural$. Then $d_N(C([0,\omega^\gamma\cdot n]),C([0,\omega^\alpha\cdot m]))\geq 2$.
\end{cor}
This corollary answers partially Problem~2 in~\cite{Gorak}.
\begin{proof}
In fact, we are going to prove the stronger claim that for $\beta<\omega^\alpha$ and for every net $\mathcal N$ in $C([0,\omega^\alpha])$ there is no Lipschitz embedding $f:\mathcal N \to C([0,\beta])$ such that $\dist(f)<2$. 

Let us suppose that such $f$ and $\mathcal N$ exist. Assume that $\mathcal N$ is an $(a,b)$-net and consider a mapping $\pi:C([0,\omega^\alpha]) \to \mathcal N$ such that $\norm{x-\pi(x)}\leq b$. 
Let us consider $M_\alpha$ as a subset of $C([0,\omega^\alpha])$ which we can by Theorem~\ref{t:classification}. Since $d(x,y)\geq 1$ for all $x\neq y \in M_\alpha$, we have, for $\lambda>2b$, that $\dist(\pi\restricted_{\lambda M_\alpha})\leq \left(1+\frac{2b}\lambda \right)\left(1+\frac{2b}{\lambda-2b}\right)$. Thus it is clear that for $\lambda$ large enough we have $\dist(g)<2$ for the embedding $g:M_\alpha \to C([0,\beta])$ defined as $g(x)=\frac1\lambda f(\pi(\lambda x))$ for $x\in M_\alpha$. According to  Theorem~\ref{t:classification}, such embedding cannot exist. Contradiction.
\end{proof}


Finally, we will give a lower bound on the Szlenk index of a Banach space $X$ that admits a certain $M_\alpha$ with distortion strictly less than $2$ (for the definition and properties of the Szlenk index, the reader can consult \cite{HMVZbook,Lancien}).

\begin{thm}\label{t:Szlenk}
Let $X$  be an Asplund space and assume that $M_{\omega^\alpha}$ embeds into $X$ with distortion strictly less than $2$ for an ordinal $\alpha < \omega_1$. Then $\Sz(X)\geq \omega^{\alpha+1}$.
\end{thm}

For the proof we will need the following version of Zippin's lemma as presented in~\cite[page 27]{Benyamini}, see also~\cite[Lemma~5.11]{Rosenthal}.
\begin{Zippin}
Let $X$ be a separable Banach space with separable dual and let $\frac12>\varepsilon>0$. Then there exist a compact $K$, an ordinal $\beta < \displaystyle \omega^{Sz(X,\frac\varepsilon8)+1}$, a subspace $Y$ of $C(K)$, isometric to $C([0,\beta])$ and an embedding $i:X \to C(K)$ with $\norm{i}\norm{i^{-1}}<1+\varepsilon$ such that for any $x\in X$ we have \[  \dist(i(x),Y)\leq 2\varepsilon\norm{i(x)}. \]
\end{Zippin}

\begin{proof}
Let us assume that $M_{\omega^\alpha} \embeds{D} X$ with $D<2$.
Let $\varepsilon>0$ be small enough so that $D'=D(1+\varepsilon)<2$ and also that for $\eta:=2\varepsilon D'$ we have $\displaystyle \frac{1+4\varepsilon }{1-2\eta}D'<2$. Let $K$ and $\ds\beta<\omega^{\Sz(X,\frac\varepsilon8)+1}$ be as in Zippin's lemma.
Then $M_{\omega^\alpha}$ embeds into $C(K)$ with distortion $D'<2$ via some embedding $g$ such that $d(x,y)\leq \norm{g(x)-g(y)}\leq D'd(x,y)$ and, without loss of generality, that $g(\basepoint)=0$. 
Thus for every $x \in M_{\omega^\alpha}$ we have $\norm{g(x)}\leq 2D'$.
We know that for each $x \in M_{\omega^\alpha}$ there is $f(x) \in C([0,\beta])$ such that $\norm{g(x)-f(x)}\leq \eta$.
This implies that $\norm{g(x)-g(y)}-2\eta\leq \norm{f(x)-f(y)}\leq \norm{g(x)-g(y)}+2\eta$.
Now since $1\leq d(x,y)$ we have 
\[
 d(x,y)(1-2\eta)\leq \norm{f(x)-f(y)}\leq d(x,y)D'(1+4\varepsilon).
\]
This proves that $f$ is a Lipschitz embedding of $M_{\omega^\alpha}$ into $C([0,\beta])$ with distortion strictly less than $2$ and so, according to Theorem~\ref{t:classification}, we have $\beta \geq \omega^{\omega^\alpha}$. 
This implies that $\Sz(X)>\omega^\alpha$ and so $\Sz(X) \geq \omega^{\alpha+1}$ by~\cite[Theorem 2.43]{HMVZbook}.
\end{proof}

An interesting immediate consequence of the above theorem is the fact that, for every $\gamma<\alpha<\omega_1$ and for every equivalent norm  $\abs{\cdot}$ on $C([0,\omega^{\omega^\gamma}])$, the space $M_{\omega^\alpha}$ does not embed with distortion strictly less than $2$ into $(C([0,\omega^{\omega^\gamma}]),\abs{\cdot})$.

\end{document}